\documentclass[11pt]{amsart}

\usepackage[utf8x]{inputenc}
\usepackage{amsmath}
\usepackage{amssymb}
\usepackage{textcomp}
\usepackage{tikz}
\usetikzlibrary{trees}
\usepackage{textcomp}
\usepackage{color}
\newtheorem{theorem}{Theorem}[section]

\newtheorem{corollary}[theorem]{Corollary}

\theoremstyle{definition}

\theoremstyle{remark}
\newtheorem{remark}[theorem]{Remark}

\numberwithin{equation}{section}

\newcommand{\rd}{{\mathbb R^d}}

\newcommand{\rr}{{\mathbb R}}

\def\R{{\mathbb R}}

\def\C{{\mathbb C}}

\def\E{{\mathbb E}}

\def\D{{\mathbb D}}

\begin{document}

\title{\bf Strong analytic solutions of fractional Cauchy problems}\thanks{Mathematics Subject Classification (2010); Primary: 35R11, 35C15, 35S05. Secondary: 47G30, 60K99.}

\author{Jebessa B. Mijena}
\address{Jebessa B. Mijena, Department of Mathematics and Statistics, 221 Parker Hall, Auburn University, Auburn, Alabama 36849, USA}
\email{jbm0018@tigermail.auburn.edu}

\author{Erkan Nane}
\address{Erkan Nane, Department of Mathematics and Statistics, 221 Parker Hall, Auburn University, Auburn, Alabama 36849, USA}
\email{nane@auburn.edu}
\urladdr{http://www.auburn.edu/$\sim$ezn0001/}

\begin{abstract}
Fractional derivatives can be used to model time delays in a
diffusion process. When the order of the fractional derivative is
distributed over the unit interval, it is useful for  modeling a
mixture of delay sources. In some special cases distributed order derivative can be used to model ultra-slow diffusion.
We extend the results of Baeumer and Meerschaert \cite{fracCauchy} in the single order fractional derivative case to distributed order   fractional derivative case. In particular, we develop the
strong analytic solutions of  distributed order fractional Cauchy problems.
\end{abstract}


\maketitle
\section{Introduction}
This paper develops strong analytic solution for  distributed order fractional Cauchy problems. Cauchy problems $\frac{\partial u}{\partial t}
=Lu$ model diffusion processes and have appeared as an essential tool for the study of dynamics of various complex stochastic
processes arising in anomalous diffusion in physics \cite{Metzler, Zaslavsky}, finance \cite{Gorenflo}, hydrology \cite{Benson}, and cell biology \cite{Saxton}. Complexity includes phenomena
such as the presence of weak or strong correlations, different sub-or super-diffusive modes, and jump effects. For example,
experimental studies of the motion of macromolecules in a cell membrane show apparent subdiffusive motion with several
simultaneous diffusive modes (see \cite{Saxton}). The simplest case $L=\Delta=\sum_{j}\partial^{2}u/\partial x^{2}_{j}$ governs
a Brownian $B(t)$ with density $u(t,x)$, for which the square root scaling $u(t,x) = t^{-1/2}u(1,t^{-1/2}x)$ pertains
\cite{Einstien}.

The fractional Cauchy problem $\partial^{\beta}u/\partial t^{\beta}=Lu$ with $0<\beta<1$ models anomalous sub-diffusion, in which
a cloud of particles spreads slower than the square root of time. When $L=\Delta$, the solution $u(t,x)$ is the density of a
time-changed Brownian motion $B(E(t))$, where the non-Markovian time change $E(t)=\inf\{\tau>0;D(\tau)>t\}$ is the inverse,
or first passage time of a stable subordinator $D(t)$ with index $\beta$. The scaling $D(ct)=c^{1/\beta}D(t)$ in law implies
$E(ct)=c^{\beta}E(t)$ in law for the inverse process, so that $u(t,x)=t^{-\beta/2}u(1,t^{-\beta/2}x).$

The process $B(E(t))$ is the long-time scaling limit of a random walk \cite{Zsolution,limitCTRW}, when the random
waiting times between
jumps belong to the $\beta$-stable domain of attraction.  Roughly speaking, a power-law distribution
of waiting times leads to a fractional time derivative in the governing equation.  Recently, Barlow and
\u Cern\'y \cite{barlow-cerny} obtained $B(E(t))$ as
the scaling limit of a random walk in a random environment. More generally, for a uniformly elliptic operator
 $L$ on a bounded domain $D\subset \rd$, under suitable technical conditions and assuming Dirichlet boundary
  conditions, the diffusion equation $\partial u/\partial t=Lu$ governs a Markov process $Y(t)$ killed at the
  boundary, and the corresponding fractional diffusion equation $\partial^\beta u/\partial t^\beta=Lu$ governs the time-changed
  process
$Y(E(t))$ \cite{m-n-v-aop}.

In some applications, waiting times between particle jumps evolve according to a more complicated process, which
cannot
 be adequately described by a single power law.  A mixture of power laws leads to a distributed-order fractional
derivative in time \cite{chechkin-et-al, mainardi-1, mainardi-2, mainardi-3, M-S-ultra, naber}.  An
important application of distributed-order diffusions
 is to model ultraslow diffusion where a plume of particles spreads at a logarithmic rate \cite{M-S-ultra, Sinai}.  This paper
 considers the distributed-order time-fractional diffusion equations with the generator $L$ of a uniformly bounded and strongly continuous semigroup in a Banach space.
 Hahn et al.\ \cite{h-k-umarov} discussed the solutions of such equations on $\rd$, and the connections with certain
 subordinated processes.  Kochubei \cite{koch3} proved strong solutions on $\rd$ for the case $L=\Delta$.
  Luchko \cite{Luckho} proved the uniqueness and continuous dependence on initial conditions on bounded domains.
  Meerschaert et al. \cite{m-n-v-jmaa} established the strong solutions of distributed order fractional Cauchy problems in bounded domains with Dirichlet boundary conditions.

When  $L$  is the generator of   a uniformly bounded  and continuous semigroup on    a Banach spaces,  Baeumer and Meerschaert  \cite{fracCauchy} showed that the solution of $$\partial^{\beta}u/\partial t^{\beta}=Lu$$ is analytic in a sectorial region.
In this paper, we
extend the results of Baeumer and Meerschaert \cite{fracCauchy} to distributed order fractional diffusion case. We follow the methods of Baeumer and Meerschaert with some crucial changes in the proof of our main results. Our proofs work for operators $L$  that are  generators of  uniformly bounded  and continuous semigroups on   Banach spaces.

The paper is organized as follows. In the next section we give a brief introduction about semigroups and their generators. In section 3, we give preliminaries on the distributed order fractional derivatives and the corresponding inverse subordinators. We state and prove our main results in section 4.
\section{Cauchy problems and Semigroups}
Let $||f||_{1}=\int|f(x)|dx$ be the usual norm on the Banach space $L^1(\rd)$ of absolutely integrable functions
$f:\rd\rightarrow \rr$. A family of linear operators $\{T(t):t\geq 0\}$ on a Banach space $X$ such that $T(0)$ is the
identity operator and $T(s+t)=T(s)T(t)$ for all $s, t \geq 0$ is called a continuous convolution semigroup. If $||T(t)f||
\leq M||f||$ for all $f\in X$ and all $t\geq 0$ then the semigroup is uniformly bounded. If $T(t_{n})f\rightarrow T(t)f$ in $X$
for all $f\in X$ whenever $t_{n}\rightarrow t$ then the semigroup is strongly continuous. It is easy to check that
$\{T(t);t\geq 0\}$ is strongly continuous if $T(t)f\rightarrow f$ in $X$ for all $f\in X$ as $t\rightarrow 0$. A strongly
continuous convolution semigroup such that $||T(t)f||\leq||f||$ for all $t\geq 0$ and all $f\in X$ is also called a Feller
semigroup.
\\
\\ For any strongly continuous semigroup $\{T(t);t>0\}$ on a Banach space $X$ we define the generator
\begin{equation}\label{generator}
Lf=\lim_{t\rightarrow 0^{+}}\frac{T(t)f-f}{t}\hspace{0.1cm}\mbox{in}\hspace{0.1cm}X
 \end{equation}
meaning that $||t^{-1}(T(t)f-f)-Lf||\rightarrow 0$ in the Banach space norm. The domain $D(L)$ of this linear operator
is the set of all $f\in X$ for which the limit in \eqref{generator} exists. The domain $D(L)$ is dense in $X$, and $L$ is
closed, meaning that if $f_{n}\rightarrow f$ and $Lf_{n}\rightarrow g$ in $X$ then $f\in D(L)$ and $Lf=g$ (see, for example
Corollary I.2.5 in \cite{pazy}).
\\
\\ Another consequence of $T(t)$ being a strongly continuous semigroup is that $u(t)=T(t)f$ solves the abstract Cauchy problem
$$\frac{d}{dt}u(t)=Lu(t);\hspace{0.3cm}u(0)=f$$
for $f\in D(L)$. Furthermore, the integrated equation $T(t)f=L\int^{t}_{0}T(s)fds+f$ holds for all $f\in X$ (see, for example,
 Theorem I.2.4 in \cite{pazy}).
\\

\section{Distributed order fractional derivatives}\label{sec3}
The Caputo fractional derivative \cite{Caputo} is defined for $0<\beta<1$ as
\begin{equation}\label{CaputoDef}
\frac{\partial^\beta u(t,x)}{\partial t^\beta}=\frac{1}{\Gamma(1-\beta)}\int_0^t \frac{\partial
u(r,x)}{\partial r}\frac{dr}{(t-r)^\beta} .
\end{equation}
Its Laplace transform
\begin{equation}\label{CaputolT}
\int_0^\infty e^{-st} \frac{\partial^\beta u(t,x)}{\partial
t^\beta}\,ds=s^\beta \tilde u(s,x)-s^{\beta-1} u(0,x)
\end{equation}
where $\tilde u(s,x) = \int_0^\infty e^{-st}u(t,x)dt$ and incorporates the initial value in the same way as the first
derivative.
The distributed order fractional derivative is
\begin{equation}\label{DOFDdef}
\D^{(\mu)}u(t,x):=\int_0^1\frac{\partial^\beta u(t,x)}{\partial
t^\beta} \mu(d\beta),
\end{equation}
where $\mu$ is a finite Borel measure with $\mu(0,1)>0$.

For a function  $u(t, x)$ continuous in  $t\geq 0$,  the Riemann-Liouville  fractional derivative of
order $0<\beta<1$ is defined by
\begin{equation}\label{R-LDef}
\bigg(\frac{\partial }{\partial t}\bigg)^\beta u(t,x)=\frac{1}{\Gamma(1-\beta)}\frac{\partial }{\partial t}
\int_0^t \frac{u(r,x)}{(t-r)^\beta} dr .
\end{equation}
Its Laplace transform
\begin{equation}\label{CaputoLT}
\int_0^\infty e^{-st} \bigg(\frac{\partial }{\partial t}\bigg)^\beta u(t,x)\,ds=s^\beta \tilde u(s,x).
\end{equation}
If $u(\cdot, x)$ is absolutely continuous on bounded intervals (e.g., if the derivative exists everywhere and is
integrable) then the Riemann-Liouville and Caputo derivatives are related by
\begin{equation}\label{captuto-R-L-derivatives}
\frac{\partial^\beta u(t,x)}{\partial t^\beta}=\bigg(\frac{\partial }{\partial t}\bigg)^\beta u(t,x)-\frac{
t^{-\beta}u(0,x)}{\Gamma (1-\beta)} .
\end{equation}
The Riemann-Liouville fractional derivative is more general, as it does not require the first derivative to exist.
It is also possible to adopt the right-hand side of \eqref{captuto-R-L-derivatives} as the definition of the Caputo
derivative, see for example Kochubei \cite{koch3}.  Then the (extended) distributed order derivative is
\begin{equation}\label{DOFDdefK}
\D_1^{(\mu)}u(t,x):=\int_0^1\left[\bigg(\frac{\partial }{\partial t}\bigg)^\beta u(t,x)-\frac{t^{-\beta}u(0,x)}
{\Gamma (1-\beta)}\right] \mu(d\beta) ,
\end{equation}
which exists for $u(t, x)$ continuous, and agrees with the usual definition \eqref{DOFDdef} when $u(t, x)$ is absolutely continuous.

Distributed order fractional derivatives are connected with random walk limits.  For each $c>0$, take a sequence of
 i.i.d.\ waiting times $(J_n^{c})$ and i.i.d.\ jumps $(Y_n^{c})$.  Let $X^{c}(n)=Y_1^{c}+\cdots+Y_n^{c}$ be
 the particle location after $n$ jumps, and $T^{c}(n)=J_1^{c}+\cdots+J_n^{c}$ the time of the $n$th jump.
  Suppose that $X^{c}(ct)\Rightarrow A(t)$ and $T^{c}(ct)\Rightarrow D_{\psi}(t)$ as $c\to\infty$, where the limits
  $A(t)$ and $D_{\psi}(t)$ are independent L\'evy processes.
The number of jumps by time $t\geq 0$ is $N^{c}_t=\max\{n\geq 0:T^{c}(n)\leq t\}$, and
 \cite[Theorem 2.1]{M-S-triangular} shows that $X^{c}(N_t^{c})\Rightarrow A(E_\psi(t))$, where
\begin{equation}\label{Etdef}
E_\psi(t)=\inf\{\tau:D_{\psi}(\tau)> t\} .
\end{equation}

A specific mixture model from \cite{M-S-ultra} gives rise to
distributed order fractional derivatives:  Let $(B_i)$, $0<B_i<1$,
be i.i.d.\ random variables such that
$P\{J_i^{c}>u|B_i=\beta\}=c^{-1}u^{-\beta}$, for $u\geq
c^{-1/\beta}$. Then $T^{c}(cu)\Rightarrow D_{\psi}(t)$, a subordinator
with ${\mathbb E}[e^{-s D_{\psi}(t)}]=e^{-t\psi(s)}$, where
\begin{equation}\label{phiWdef}
\psi(s)=\int_0^\infty(e^{-s x}-1)\phi(dx) .
\end{equation}
Then the associated L\'evy measure is
\begin{equation}\label{psiWdef}
\phi(t,\infty)=\int_0^1 t^{-\beta}\nu(d\beta),
\end{equation}
where $\nu$ is the distribution of $B_i$.  An easy computation gives
\begin{equation}\begin{split}\label{psiW}
\psi(s)
&= \int_0^1  s^\beta \Gamma(1-\beta) \nu(d\beta)=\int_0^1  s^\beta \mu(d\beta) .
\end{split}\end{equation}
Here we define $\mu(d\beta)=\Gamma(1-\beta) \nu(d\beta)$.
Then, Theorem 3.10 in \cite{M-S-ultra} shows that
$c^{-1}N^{c}_t\Rightarrow E_\psi(t)$, where $E_\psi(t)$ is given by
\eqref{Etdef}.
 The L\'evy process $A(t)$ defines a strongly continuous convolution semigroup with generator $L$, and $A(E_\psi(t))$ is the
 stochastic
solution to the distributed order-fractional diffusion equation
\begin{equation}\label{DOFCPdef}
\D^{(\mu)} u(t,x)=L u(t,x),
\end{equation}
where $\D^{(\mu)}$ is given by \eqref{DOFDdef} with $\mu(d\beta)=\Gamma(1-\beta)\nu(d\beta)$.
The condition
\begin{equation}\label{finite-mu-bound}
\int_0^1  \frac 1{1-\beta}\, \nu(d\beta)<\infty
\end{equation}
is imposed to ensure that $\mu(0,1)<\infty$. Since
$\phi(0,\infty)=\infty$ in \eqref{phiWdef}, Theorem 3.1
 in \cite{M-S-triangular} implies that $E_\psi(t)$ has a Lebesgue density
\begin{equation}\label{E-lebesgue-density}
f_{E_\psi(t)}(x)=\int_0^t \phi(t-y,\infty) P_{D_\psi(x)}(dy) .
\end{equation}
Note that $E_\psi(t)$ is almost surely continuous and  nondecreasing.

 \section{Main Results}

 Let $D_\psi(t)$ be a
strictly increasing L\'evy process (subordinator) with ${\mathbb E}[e^{-s D_\psi(t)}]=e^{-t\psi(s)}$, where the Laplace exponent
\begin{equation}\label{psiD2}
\psi(s)=bs+\int_0^\infty(e^{-s x}-1)\phi(dx) ,
\end{equation}
$b\geq 0$, and $\phi$ is the L\'evy measure of $D_\psi$.  Then we must have either
\begin{equation}\label{A1}
\phi(0,\infty)=\infty ,
\end{equation}
or $b>0$, or both, see \cite{M-S-triangular}.  Let
\begin{equation}\label{Epsi-def}
E_\psi(t)=\inf\{\tau\geq 0:\ D_{\psi}(\tau)>t\}
\end{equation}
be the inverse subordinator.

Let $T$ be a uniformly bounded, strongly
continuous semigroup on a Banach space.
 Let
 \begin{equation}\label{eq-S}
 S(t)f=\displaystyle \int^{\infty}_{0}(T(l)f)f_{E_\psi(t)}(l)dl
\end{equation}
 where $f_{E_\psi(t)}(l)$ is a Lebesgue density of $E_\psi(t)$.

 Using \eqref{E-lebesgue-density}, it is easy to show that
  $$ \int^{\infty}_{0}e^{-st}f_{E_\psi(t)}(l)dt=\frac{1}{s}\psi(s)e^{-l\psi(s)}.$$
Using Fubini's Theorem, we get
\begin{equation}\label{eq-laplace}
\int^{\infty}_{0}\psi(s)e^{-l\psi(s)}T(l)fdl=s\int^{\infty}_{0}e^{-st}S(t)fdt.
\end{equation}

We define a {\it sectorial region} of the complex plane
$\C(\alpha)=\{re^{i\theta}\in \C:\ r>0, |\theta|<\alpha \}.$
We call a family of linear operators on a Banach space $X$ {\it strongly analytic in a sectorial region} if for some $\alpha>0$ the mapping $t\to T(t)f$ has an analytic extension to the sectorial region $\C(\alpha)$ for all $f\in X$ (see, for example, section 3.12 in \cite{HiPh}).

Let $0<\beta_{1}<\beta_{2}<\cdot\cdot\cdot < \beta_{n}<1$. In the next theorem we consider the case where $$\psi(s)=c_{1}s^{\beta_{1}}+c_{2}s^{\beta_{2}}+\cdots +c_{n}s^{\beta_{n}}.$$
In this case the L\'evy subordinator can be written as
$$D_\psi(t)=(c_1)^{1/\beta_1}D^1(t)+(c_2)^{1/\beta_2}D^2(t)+\cdots+ (c_n)^{1/\beta_n}D^n(t)$$
where $D^1(t), D^2(t),\cdots , D^n(t)$ are independent  stable subordinators of index $0<\beta_{1}<\beta_{2}<\cdot\cdot\cdot < \beta_{n}<1$.

 \begin{theorem}\label{theorem1} Let $(X,||.||)$ be a Banach space and $L$ be the generator of  a uniformly bounded, strongly
continuous semigroup $\{T(t):t\geq 0\}$. Then the family $\{S(t):t\geq 0\}$ of linear operators from $X$ into $X$ given by
\eqref{eq-S} is uniformly bounded and strongly analytic in a sectorial region. Furthermore,
$\{S(t): t\geq 0\}$ is strongly continuous and $g(x, t)= S(t)f(x)$ is a solution of
$$\displaystyle \sum_{i=1}^{n}c_{i}\frac{\partial^{\beta_{i}}g(x,t)}{\partial t^{\beta_{i}}}
=Lg(x, t); \ g(x,0)=f(x).$$
$\hspace{0.2cm}\mbox{for}\hspace{0.2cm}\beta_{1}<\beta_{2}<\cdot\cdot\cdot < \beta_{n}\in (0, 1)$
\end{theorem}

\begin{proof}We adapt the methods of  Baeumer and  Meerschaert \cite[Theorem 3.1]{fracCauchy} with some very crucial changes in the following. For the purpose of completeness of the arguments we included some parts verbatim from Baeumer and Meerschaert \cite{fracCauchy}.

 Since $\{T(t):t\geq 0\}$ is uniformly bounded we have $||T(t)f||\leq M||f||$ for all $f\in X$.
Bochner's Theorem (\cite[Thm. 1.1.4]{ABHN}) implies that a function $F:\mathbb{R}^{1}\rightarrow X$ is integrable if and only if
$F(s)$ is measurable and $||F(s)||$ is integrable, in which case
$$||\int F(l)dl||\leq \int||F(l)||dl.$$
For fixed $f\in X$ and applying Bochner's Theorem with $F(l)=(T(l)f)f_{E_\psi(t)}(l)$ we  have that
\begin{eqnarray*}
 ||S(t)f|| &=& ||\int^{\infty}_{0}(T(l)f)f_{E_\psi(t)}(l)dl||\\ &\leq& \int^{\infty}_{0} ||(T(l)f)f_{E_\psi(t)}(l)||dl\\ &=&
\int^{\infty}_{0}||T(l)f||f_{E_\psi(t)}(l)dl\\ &\leq& \int^{\infty}_{0}M||f||f_{E_\psi(t)}(l)dl=M||f||
\end{eqnarray*}
since $f_{E_\psi(t)}(l)$ is the Lebesgue density for $E_\psi(t)$. This shows that $\{S(t):t\geq 0\}$ is well defined and uniformly
bounded family of linear operators on $X$.
\\
\\
\\ The definition of $T(t)$ and dominated convergence theorem implies
\begin{eqnarray*}
 ||S(t)f-f|| &=& ||\displaystyle \int^{\infty}_{0}(T(l)f-f)f_{E_\psi(t)}(l)dl|| \\ &\leq&
\displaystyle \int^{\infty}_{0}||T(l)f-f||f_{E_\psi(t)}(l)dl\\ &\to &||T(0)f-f||=0
\end{eqnarray*}
as $t\rightarrow 0^{+}$. This shows $\displaystyle \lim_{t\rightarrow 0^{+}}S(t)f=f$. Now if $t, h>0$ then we have
$$||S(t+h)f-S(t)f||\leq \int^{\infty}_{0}||T(l)f|||f_{E_\psi(t+h)}(l)-f_{E_\psi(t)(l)}|dl\rightarrow 0$$ as
$h\rightarrow 0^{+}$ since
 $E_\psi(t+h)\implies E_\psi(t)$ as $h\to 0$.

 This shows that $\{S(t):t>0\}$ is strongly continuous.
\\
\\ Let $q(s)=\int^{\infty}_{0}e^{-st}T(t)fdt$ and $r(s)=\int^{\infty}_{0}e^{-st}S(t)fdt$ for any $s>0,$ so that we can write
\eqref{eq-laplace} in the form
\begin{equation}\label{eq-32}
\psi(s)q(\psi(s))=sr(s)
\end{equation}
for any $s>0$. Now we want to show that this relation holds for certain complex numbers. Fix $s \in \C_{+}=\{z\in \C:\mathcal{R}(z)> 0\}$,
 and let $F(t)=e^{-st}T(t)f$. Since $F$ is continuous, it is measurable, and we have $||F(t)||\leq |e^{-st}|M||f||=e^{-t
\mathcal{R}(s)}
M||f||$ since $||T(t)f||\leq M||f||,$ so that the function $||F(t)||$ is integrable. Then Bochner's Theorem implies that
$q(s)=\int_0^\infty F(t)dt$ exists for all $s\in \C_{+},$ with
\begin{equation}\label{eq-33}
||q(s)||=||\int_0^\infty F(t)dt||\leq \int_0^\infty ||F(t)||dt\leq \int_0^\infty e^{-t\mathcal{R}(s)}M||f|| dt=\frac{M||f||}{\mathcal{R}(s)}.
\end{equation}
Since $q(s)$ is the Laplace transform of the bounded continuous function $t\mapsto T(t)f$, Theorem 1.5.1 of $[1]$ shows that
$q(s)$ is an analytic function on $s\in \C_{+}$.
\\
\\
 Now we carry out the details of the proof for only in the case $n=2$.
 We want to show that $r(s)$ is the Laplace transform of an analytic function defined on a sectorial region. Theorem
2.6.1 of \cite{ABHN} implies that if for some real $x$ and some $\alpha \in (0, \pi/2]$ the function $r(s)$ has an analytic extension
to the region $x+\C(\alpha+\pi/2)$ and if $\sup\{||(s-x)r(s)||:s\in x+\C(\alpha'+\pi/2)\}<\infty$ for all $0<\alpha'<\alpha$,
then there exists an analytic function $\overline{r}(t)$ on $t\in \C(\alpha)$ such that $r(s)$ is the Laplace transform of
$\overline{r}(t)$. We will apply the theorem with $x=0$. It follows from (\ref{eq-32}) that $r(s)=\frac{1}{s}\psi(s)q(\psi(s))$
 for all $s>0$, but the right hand side here is well defined and analytic on the set of complex $s$ that are not on the branch
cut and are such that $\mathcal{R}(\psi(s))=\mathcal{R}(c_{1}s^{\beta_{1}}+c_{2}s^{\beta_{2}})>0$, since
$\beta_{1}<\beta_{2}$, it suffices to consider $\mathcal{R}(s^{\beta_{2}})>0$, so if $1/2<\beta_{2}<1$, then $r(s)$ has a unique analytic extension to the sectorial region
$\C(\pi/2\beta_{2})=\{s\in \C:\mathcal{R}e(s^{\beta_{2}})>0\}$ (e.g., \cite[3.11.5]{HiPh} ), and note that $\pi/2\beta_{2}=\pi/2+\alpha$
for some $\alpha>0$. If $\beta_{2}<1/2$ then $r(s)$ has an analytic extension to the sectorial region $s\in \C(\pi/2+\alpha)$
for any $\alpha < \pi/2$ and $\mathcal{R}(s^{\beta_{2}})>0$ for all such $s$. Now for any complex $s=re^{i\theta}$ such that
$s\in \C(\pi/2+\alpha')$ for any $0<\alpha'<\alpha$, we have in view of \eqref{eq-32} and \eqref{eq-33} that
\begin{eqnarray}
 ||sr(s)|| &=& ||\psi(s)q(\psi(s))|| \nonumber\\
  &=& |c_{1}s^{\beta_{1}}+c_{2}s^{\beta_{2}}|||q(c_{1}
s^{\beta_{1}}+c_{2}s^{\beta_{2}})||\nonumber \\
&=&\left| \frac{c_{1}r^{\beta_{1}}e^{i\beta_{1}\theta}+c_{2}r^{\beta_{2}}
e^{i\beta_{2}\theta}}{c_{1}r^{\beta_{1}}\cos(\beta_{1}\theta)+c_{2}
r^{\beta_{2}}\cos(\beta_{2}\theta)}\right|\nonumber\\
& & \ \ \times||\mathcal{R}
(c_{1}s^{\beta_{1}}+c_{2}s^{\beta{2}})q(c_{1}s^{\beta_{1}}+c_{2}s^{\beta_{2}})||\nonumber\\
 &\leq & \bigg|
\frac{c_{1}r^{\beta_{1}}e^{i\beta_{1}\theta}}{c_{1}r^{\beta_{1}}\cos(\beta_{1}\theta)+
c_{2}r^{\beta_{2}}\cos(\beta_{2}\theta)}
\bigg| M||f|| \nonumber\\
 && +\left|
\frac{c_{2}r^{\beta_{2}}e^{i\beta_{2}\theta}}{c_{1}r^{\beta_{1}}
\cos(\beta_{1}\theta)+c_{2}r^{\beta_{2}}\cos(\beta_{2}\theta)}
\right| M||f|| \nonumber\\
 &\leq& \left (\frac{1}{cos(\beta_{1}\theta)}+\frac{1}{\cos(\beta_{2}\theta)}\right)M||f||  \label{eq-34}
\end{eqnarray}
which is finite since $|\beta_{1}\theta|<|\beta_{2}\theta|<\pi/2$. Hence Theorem 2.6.1 of \cite{ABHN} implies there exists an
analytic function $\overline{r}(t)$ on $t\in \C(\alpha)$ with Laplace transform $r(s)$. Using the uniqueness of the Laplace
transform (e.g., \cite[Thm. 1.7.3]{ABHN}), if follows that $t\mapsto S(t)f$ has an analytic extension (namely $t\mapsto \overline{r}
(t)$) to the sectorial region $t\in \C(\alpha)$. Next we wish to apply Theorem 2.6.1 of \cite{ABHN}  again to show that for any
$0<\beta_{1}<\beta_{2}<1$ the function
\begin{equation}\label{eq-35}
t\mapsto \int^{t}_{0}\frac{(t-u)^{-\beta_{i}}}{\Gamma(1-\beta_{i})}S(u)fdu\hspace{1cm}i=1,2
\end{equation}
has an analytic extension to the same sectorial region $t\in \C(\alpha)$. It is easy to show  that
\begin{equation}\label{eq-36}
\int^{\infty}_{0}\frac{t^{-\beta_{i}}}{\Gamma(1-\beta_{i})}e^{-st}dt=s^{\beta_{i}-1}
\end{equation}
for any $0<\beta_{i}<1$ and any $s>0$. Since $r(s)$ is the Laplace transform of $t\mapsto S(t)f$, it follows from the
convolution property of the Laplace transform (e.g. property 1.6.4 \cite{ABHN}) that the function \eqref{eq-35} has Laplace transform
$s^{\beta_{i}-1}r(s)$ for all $s>0$. Since $r(s)$ has an analytic extension to the sectorial region $s\in \C(\pi/2+\alpha)$,
so does $s^{\beta_{i}-1}r(s)$. For any $x>0$, if $s=x+re^{i\theta}$ for some $r>0$ and $|\theta|<\pi/2+\alpha'$ for any
$0<\alpha'<\alpha$ then in view of \eqref{eq-34} we have
\begin{eqnarray*}
 ||(s-x)s^{\beta_{i}-1}r(s)|| &=& ||(s-x)s^{\beta_{i}-2}sr(s)||\\ &\leq& r||s^{\beta_{i}-2}||\left(\frac{1}{\cos(\beta_{1}
\theta)}+\frac{1}{\cos(\beta_{2}\theta)}\right)M||f||
\end{eqnarray*}
 where $||s||$ is bounded away from zero, $||s||\leq r+x$ and $\beta_{i}-2<-1$, so that $||(s-x)s^{\beta_{i}-1}r(s)||$ is
bounded on the region $x+\C(\alpha'+\pi/2)$ for all $0<\alpha'<\alpha$. Then it follows as before that the function \eqref{eq-35} has
an analytic extension to the sectorial region $t\in \C(\alpha)$.
\\
\\ Since $\{T(t):t\geq 0\}$ is a strongly continuous semigroup with generator $L$, Theorem 1.2.4 (b) in $[16]$ implies that
$\int^{t}_{0}T(s)fds$ is in the domain of the operator $L$ and
$$T(t)f= L\int^{t}_{0} T(s)fds+f.$$
Since the Laplace transform  $q(s)$ of $t\mapsto T(t)f$ exists, Corollary 1.6.5 of $[1]$ show that the Laplace transform of
$t\mapsto \int^{t}_{0}T(s)fds$ exists and equals $s^{-1}q(s)$. Corollary 1.2.5 \cite{pazy} shows that $L$ is closed. Fix $s$ and let
$g=q(s)=\int^{\infty}_{0}e^{-st}T(t)fdt$ and let $g_{n}$ be a finite Riemann sum approximating this integral, so that
$g_{n} \rightarrow g$ in $X$. Let $h_{n}=s^{-1}g_{n}$ and $h=s^{-1}g$. Then $g_{n},g$ are in the domain of $L,g_{n}
\rightarrow g$ and $h_{n}\rightarrow h$. Since $h_{n}$ is a finite sum we also have $L(h_{n})=s^{-1}L(g_{n})
\rightarrow s^{-1}L(g)$. Since $L$ is closed, this implies that $h$ is in the domain of $L$ and that $L(h)=s^{-1}L(g)$.
In other words, the Laplace transform of $t\mapsto L\int^{t}_{0}T(s)fds$ exists and equals $s^{-1}Lq(s)$. Then we have by
taking the Laplace transform of each term
$$\int^{\infty}_{0}e^{-sl}T(l)fdt=s^{-1}L\int^{\infty}_{0}e^{-sl}T(l)fdl+s^{-1}f$$
for all $s>0$. Multiply through by $s$ to obtain
$$s\int^{\infty}_{0}e^{-sl}T(l)fdl=L\int^{\infty}_{0}e^{-sl}T(l)fdl+f$$
and substitute $c_{1}s^{\beta_{1}}+c_{2}s^{\beta_{2}}$ for $s$ to get
\\
\\ $(c_{1}s^{\beta_{1}}+c_{2}s^{\beta_{2}})\int^{\infty}_{0}e^{-(c_{1}s^{\beta_{1}}+c_{2}s^{\beta_{2}})l}T(l)fdl=
L\int^{\infty}_{0}e^{-(c_{1}s^{\beta_{1}}+c_{2}s^{\beta_{2}})l}T(l)fdl+f$
\\
\\ for all $s>0$. Now use \eqref{eq-laplace}  twice to get
$$s\int^{\infty}_{0}e^{-sl}S(l)fdl=L\bigg( \frac{s}{c_{1}s^{\beta_{1}}+c_{2}s^{\beta_{2}}}\int^{\infty}_{0}
e^{-sl}S(l)fdl\bigg)+f $$
\\ and multiplying both sides by $\hspace{0.2cm}c_{1}s^{\beta_{1}-2}+c_{2}s^{\beta_{2}-2}$ we get
\begin{equation}\label{eq-37}
(c_{1}s^{\beta_{1}-1}+c_{2}s^{\beta_{2}-1})\int^{\infty}_{0}e^{-sl}S(l)fdl=Ls^{-1}
\int^{\infty}_{0}e^{-st}S(l)fdl+ c_{1}s^{\beta_{1}-2}f+c_{2}s^{\beta_{2}-2}f.
\end{equation}
 where we have again used the fact that $L$ is closed. The term on the left hand side of \eqref{eq-37} is $c_{1}s^{\beta_{1}-1}r(s)
+c_{2}s^{\beta_{2}-1}r(s)$ which was already shown to be the Laplace transform of the function $c_{1}\int^{t}_{0}
\frac{(t-u)^{-\beta_{1}}}{\Gamma(1-\beta_{1})}S(u)fdu+c_{2}\int^{t}_{0}
\frac{(t-u)^{-\beta_{2}}}{\Gamma(1-\beta_{2})}S(u)fdu$,
 which is analytic in a sectorial region. Equation \eqref{eq-36} also shows that $s^{\beta_{i}-2}$ is the Laplace transform of
$t\mapsto \frac{t^{1-\beta_{i}}}{\Gamma(2-\beta)}$. Now take the term $c_{1}s^{\beta_{1}-2}f+c_{2}s^{\beta_{2}-2}f$ to the
other side and invert the Laplace transforms. Using the fact that $\{S(t):t\geq 0\}$ is uniformly bounded, we can apply the
Phragmen-Mikusinski Inversion formula for the Laplace transform (see \cite[Corollary 1.4]{baeumer-03}) to obtain
\\
\\ $c_{1}\left(\int^{t}_{0}\frac{(t-u)^{-\beta_{1}}}{\Gamma(1-\beta_{1})}S(u)fdu-\frac{t^{1-\beta_{1}}}
{\Gamma(2-\beta_{1})}f\right)+c_{2}\left(\int^{t}_{0}\frac{(t-u)^{-\beta_{2}}}{\Gamma(1-\beta_{2})}S(u)fdu-\frac{t^{1-\beta_{2}}}
{\Gamma(2-\beta_{2})}f\right)$
\\ $=\displaystyle \lim_{n\rightarrow \infty}L \sum^{N_{n}}_{j=1}\alpha_{n,j}\frac{e^{c_{nj}l}}{c_{nj}}
\int^{\infty}_{0}e^{-c_{nj}l}S(l)fdl$
\\
\\ where the constants $N_{n}, \alpha_{n,j}$, and $c_{nj}$ are given by the inversion formula and the limit is uniform on
compact sets. Using again the fact that $L$ is closed we get
\\
\\
\begin{eqnarray}
&&c_{1}\left(\int^{t}_{0}\frac{(t-u)^{-\beta_{1}}}{\Gamma(1-\beta_{1})}S(u)fdu-\frac{t^{1-\beta_{1}}}{\Gamma(2-\beta_{1})}f
\right)\nonumber\\
&& + c_{2}\left(\int^{t}_{0}\frac{(t-u)^{-\beta_{2}}}{\Gamma(1-\beta_{2})}S(u)fdu-\frac{t^{1-\beta_{2}}}
{\Gamma(2-\beta_{2})}f\right)\nonumber\\
&&=L\int^{t}_{0}S(l)fdl\label{eq-38}
\end{eqnarray}
\\ and since the function \eqref{eq-35} is analytic in a sectorial region, the left hand side of \eqref{eq-38} is differentiable for $t>0$
Corollary 1.6.6 of \cite{ABHN} shows that
\begin{equation}\label{eq-39}
\frac{d}{dt}\int^{t}_{0}\frac{(t-u)^{-\beta_{i}}}{\Gamma(1-\beta_{i})}S(u)fdu
\end{equation}
has Laplace transform $s^{\beta_{i}}r(s)$ and hence \eqref{eq-39} equals $\frac{d^{\beta_{i}}S(t)f}{dt^{\beta_{i}}}$. Now take the
derivative with respect to $t$ on both sides of \eqref{eq-38} to obtain
$$c_{1}\left(\frac{d^{\beta_{1}}}{dt^{\beta_{1}}}S(t)f-\frac{t^{-\beta_{1}}}{\Gamma(1-\beta_{1})}f\right)+c_{2}
\left(\frac{d^{\beta_{2}}}{dt^{\beta_{2}}}S(t)f-\frac{t^{-\beta_{2}}}{\Gamma(1-\beta_{2})}f\right)=LS(t)f$$
for all $t>0$, where we use the fact that $L$ is closed to justify taking the derivative inside. Using the relation \eqref{captuto-R-L-derivatives} between the Rieman-Liuoville and Caputo fractional derivatives we proved the theorem
\end{proof}

The next theorem provides an extension with subordinator $D_\mu(t)$ as the weighted average of an arbitrary number of independent stable subordinators.
Let $E_\mu(t)$ be the inverse of the subordinator $D_\mu(t)$ with Laplace exponent $\psi(s)=\int_0^1s^\beta d\mu(\beta)$ where $supp \mu\subset (0,1)$.

\begin{theorem}
Let $(X, ||\cdot||)$ be a Banach space and $\mu$ be  a positive finite measure with supp
$\mu \subset (0, 1)$. Then the family $\{S(t):t\geq 0\}$ of linear operators from $X$ into $X$ given by $S(t)f=\displaystyle
\int^{\infty}_{0}(T(l)f)f_{E_\mu(t)}(l)dl$, is uniformly bounded and strongly analytic in a sectorial region. Furthermore,
$\{S(t):t\geq 0\}$ is strongly continuous and $g(x, t)=S(t)f(x)$ is a solution of
\begin{equation}\label{distributed-frac-prob}
\int^{1}_{0}\partial^{\beta}_{t}g(x,t)\mu(d\beta)= Lg(x,t); \ g(x,0)=f(x).
\end{equation}
\end{theorem}
\begin{proof} Since supp $\mu \subset (0, 1)$, the density $f_{E_\mu(t)}(l),l\geq 0$, exists and since $||T(l)f||
\leq M||f||$, then $S(t)f$ exists and $||S(t)f||\leq M||f||$. Also, $S(t)f$ is strongly continuous as in Theorem \ref{theorem1}.

 Let $q(s)=  \int^{\infty}_{0}e^{-st}T(t)fdt$ and $r(s)= \int^{\infty}_{0}e^{-st}S(t)fdt$ for any
$s>0$, then by \eqref{eq-laplace} we have
\begin{equation}\label{thm2-eq41}
\psi(s)q(\psi(s))=sr(s)\ \  \mathrm{where}\ \ \psi(s)=
\int^{1}_{0}s^{\beta}\mu(d\beta)
\end{equation}
for any $s>0$. Now we want to show that this relation holds for certain complex numbers $s$. In Theorem \ref{theorem1}, we have shown
that $q(s)$ is an analytic function on $s\in \C_{+}$ and $||q(s)||\leq  \frac{M||f||}{\mathcal{R}(s)}$.

 Now we want to show that $r(s)$ is the Laplace transform of an analytic function defined on a sectorial region. It follows
from equation \eqref{thm2-eq41} that
 $$r(s)=\left( \int^{1}_{0}s^{\beta-1}\mu(d\beta)\right)q\left( \int^{1}_{0}s^{\beta}\mu(d\beta)\right)$$
for all $s>0$, but the right hand side here is well defined and analytic on the set of complex $s$ such that $\mathcal{R}
\left( \int^{1}_{0}s^{\beta}\mu(d\beta)\right)>0$. Let $\beta_{1}=\sup\{supp \hspace{0.1cm}\mu\}$ 
 and fix $\epsilon >0$  small such that
$\frac{\pi/2-\epsilon}{\beta_{1}}>\pi/2$. So if $1/2< \beta_{1}<1$, then $r(s)$ has a unique analytic extension to the sectorial
region $\C(\frac{\pi/2-\epsilon}{\beta_{1}})\subset \{s\in \mathbb{C}: \mathcal{R}( \int^{1}_{0}s^{\beta}\mu(d\beta))
>0\}$ and note that $\frac{\pi/2-\epsilon}{\beta_{1}}=\pi/2+\alpha$ for some $\alpha> 0$. If $0<\beta_{1}<1/2$ then $r(s)$
has an analytic extension to the sectorial region $s\in \C(\pi/2+\alpha)$ for any $\alpha<\pi/2$, and $\mathcal{R}
(\int^{1}_{0}s^{\beta}\mu(d\beta))>0$ for all such $s$. Now for any complex $s=re^{i\theta}$ such that $s\in
\C(\pi/2+\alpha^{'})$ for any $0<\alpha^{'}<\alpha$ we have that
\begin{eqnarray*}
 ||sr(s)|| &=& \left|\left|\left(\displaystyle \int^{1}_{0}s^{\beta}\mu(d\beta)\right) q \left(\displaystyle \int^{1}_{0}s^{\beta}\mu(d\beta)\right)\right|\right| \\
&\leq& \left|\displaystyle \frac{\displaystyle \int^{1}_{0}s^{\beta}\mu(d\beta)}{\mathcal{R}
\left(\displaystyle \int^{1}_{0}s^{\beta}\mu(d\beta\right)}\right|M||f|| \\ &=& \left|\frac{\int^{1}_{0}r^{\beta}\cos(\beta \theta)
\mu(d\beta) + i\int^{1}_{0}r^{\beta}\sin(\beta \theta)\mu(d\beta)}{\int^{1}_{0}r^{\beta}\cos(\beta \theta)\mu(d\beta)}
\right|M||f||
\\ &\leq& \left(1+\left|\frac{\displaystyle \int^{1}_{0}r^{\beta}\sin(\beta\theta)\mu(d\beta)}{\displaystyle \int^{1}_{0}r^{\beta}
\cos(\beta\theta)\mu(d\beta)}\right|\right)M||f|| \\ &\leq& \left(1+\displaystyle \frac{\displaystyle \int^{1}_{0}r^{\beta}
(d\beta)}{\cos(\pi/2-\epsilon)\displaystyle \int^{1}_{0}r^{\beta}\mu(d\beta)}\right)M||f||
 \\ &=& \left(1+\displaystyle
\frac{1}{\cos(\pi/2-\epsilon)}\right)M||f||<\infty.
\end{eqnarray*}
Hence theorem 2.6.1 of \cite{ABHN} implies  that there exists an analytic function $\overline{r}(t)$ on $t\in \C(\alpha)$ with Laplace
transform $r(s)$. Using the uniqueness of the Laplace transform it follows that $t\mapsto S(t)f$ has an analytic extension
$\overline{r}(t)$ to the sectorial region $t\in \C(\alpha)$.
\\
\\ As in  Theorem \ref{theorem1}  for any $\beta\in supp \mu$ the function
\begin{equation}\label{thm2-eq-35}
t\mapsto \int^{t}_{0}\frac{(t-u)^{-\beta}}{\Gamma(1-\beta)}S(u)fdu
\end{equation}
 has analytic extension to the sectorial region $t\in \C(\alpha)$.
Next we wish to apply Theorem 2.6.1 of \cite{ABHN} again to show that for any $0< \beta <1$ the function
\begin{equation}\label{thm2-eq-39}
t\mapsto \int^{t}_{0}\bigg(\int_0^1\frac{(t-u)^{-\beta}}{\Gamma(1-\beta)}\mu(d\beta)\bigg)S(u)fdu
\end{equation}
has analytic extension to the sectorial region $t\in \C(\alpha)$.
\\
\\ Since $\displaystyle \int^{\infty}_{0}\left(\displaystyle \int^{1}_{0}\displaystyle \frac{t^{-\beta}}{\Gamma(1-\beta)}\mu
(d\beta)\right)e^{-st}dt=\displaystyle \int^{1}_{0}s^{\beta-1}\mu(d\beta)$
\\
\\ for any $0<\beta<1$ and any $s>0$ and $r(s)$ is the Laplace transform of $t\mapsto S(t)f$ it follows from convolution property
of the Laplace transform that the function  \eqref{thm2-eq-39} has Laplace transform $s^{-1}\psi(s)r(s)$ for all $s>0$. Since $r(s)$ has an
analytic extension to the sectorial region $s\in \C(\pi/2+\alpha)$, so does $s^{-1}\psi(s)r(s)$. For any $x>0$, if
$s=x+re^{i\theta}$ for some $r>0$ and $|\theta|<\pi/2+\alpha^{'}$ for any $0<\alpha^{'}<\alpha$ then we have
\begin{eqnarray*}
 ||(s-x)\left(\int^{1}_{0}s^{\beta-1}\mu(d\beta)\right)r(s)|| &=& ||(s-x)\int^{1}_{0}s^{\beta-2}s.r(s)\mu(d\beta)|| \\ &\leq&
r||\int^{1}_{0}s^{\beta-2}\mu(d\beta)||(1+\frac{1}{\cos(\pi/2-\epsilon)})M||f|| \\ &\leq&r\left(\int^{1}_{0}||s^{\beta-2}||
\mu(d\beta)\right)\left(1+\frac{1}{\cos(\pi/2-\epsilon)}\right)M||f|| \\ &\leq& r\left(\int^{1}_{0}x^{\beta-2}\mu(d\beta)\right)
\left(1+\frac{1}{\cos(\pi/2-\epsilon)}\right)M||f|| \\ &=& \frac{r}{x^{2}}\left(\int^{1}_{0}x^{\beta}\mu(d\beta)\right)
\left(1+\frac{1}{\cos(\pi/2-\epsilon)}\right)M||f||.
\end{eqnarray*}
Since $\mu$ positive finite measure and $x>0$, so that $||(s-x)s^{-1}\psi(s)r(s)||$ is bounded on the region $x+\C(\alpha^{'}
+\pi/2)$ for all $0<\alpha^{'}<\alpha$. Then it follows as before that the function \eqref{thm2-eq-39} has an analytic extension to the
sectorial region $\C(\alpha)$.

Since $\{T(t):t\geq 0\}$ is a strongly continuous semigroup with generator $L$, Theorem 1.2.4 (b) in \cite{pazy} implies that
$\int^{t}_{0}T(l)fdl$ is in the domain of the operator $L$ and
$$T(t)f=L \int^{t}_{0}T(l)fdl +f.$$
Then by taking Laplace transform of both sides we have
$$\int^{\infty}_{0}e^{-st}T(t)fdt=s^{-1}L\int^{\infty}_{0}e^{-st}T(t)fdt +s^{-1}f$$
for all $s>0$. Multiply both sides by $s$ to obtain
$$s\int^{\infty}_{0}e^{-st}T(t)fdt=L\int^{\infty}_{0}e^{-st}T(t)fdt +f$$
and substitute $\psi(s)=\int^{1}_{0}s^{\beta}\mu(d\beta)$ for $s$ to obtain
$$\psi(s)\int^{\infty}_{0}e^{-\psi(s)t}T(t)fdt=L\int^{\infty}_{0}e^{-\psi(s)t}T(t)fdt +f$$
for all $s>0$. Now use \eqref{thm2-eq41} twice to get
$$s\int^{\infty}_{0}e^{-st}S(t)fdt=L\frac{s}{\psi(s)}\int^{\infty}_{0}e^{-st}S(t)fdt +f$$
and multiplying through by $s^{-2}\psi(s)$ to get
$$s^{-1}\psi(s)\int^{\infty}_{0}e^{-st}S(t)fdt=Ls^{-1}\int^{\infty}_{0}e^{-st}S(t)fdt +\psi(s)s^{-2}f$$
since $L$ is closed. Invert the Laplace transform to get
\begin{equation}
\begin{split}
&\int^{t}_{0}\left(\int^{1}_{0}\frac{(t-u)^{-\beta}}{\Gamma(1-\beta)}\mu(d\beta)\right)S(u)fdu-\int^{1}_{0}\frac{t^{1-\beta}}{\Gamma(2-\beta)}f\mu
(d\beta)\\
&=\lim_{n\rightarrow \infty} L\sum^{N_{n}}_{j=1}\alpha_{n,j}\frac{e^{c_{n_{j}}t}}{c_{n_{j}}}
\int^{\infty}_{0}e^{-C_{n_{j}}t}S(t)fdt
\end{split}
\end{equation}
where the constant $N_{n},\alpha_{n,j}$, and $c_{n}$ are given by the inversion formula.

Next using Fubini's theorem
we show that $\int^{1}_{0}\int^{t}_{0}\frac{(t-u)^{-\beta}}{\Gamma(1-\beta)}S(u)fdu\mu(d\beta)$ have
same Laplace transform $s^{-1}\psi(s)r(s).$ This is true because
\begin{equation}
\begin{split}
 &\int^{1}_{0}\int^{\infty}_{0}||e^{-st}\int^{t}_{0}\frac{(t-u)^{-\beta}}{\Gamma(1-\beta)}S(u)f||dudt\mu(d\beta)\\
&\leq M||f||\int^{1}_{0}\int^{\infty}_{0}e^{-st}\int^{t}_{0}\frac{(t-u)^{-\beta}}{\Gamma(1-\beta)}dudt\mu(d\beta)\\
&= M||f||\int^{1}_{0}\int^{\infty}_{0}\frac{e^{-st}t^{1-\beta}}{\Gamma(2-\beta)}dt\mu(d\beta)\\
&\leq M||f||\int^{1}_{0}s^{\beta-2}\mu(d\beta)<\infty.
\end{split}
\end{equation}
Since $\mu$ is positive finite measure and $S(t)f$ is uniformly bounded then using Fubini's theorem and the uniqueness of the Laplace transform for functions in $L_{loc}^1(\rd)$ (Theorem 1.7.3 in \cite{ABHN})
we have
\begin{equation}
\begin{split}
&\int^{1}_{0}\left[\int^{t}_{0}\frac{(t-u)^{-\beta}}{\Gamma(1-\beta)}S(u)fdu-\frac{t^{1-\beta}}{\Gamma(2-\beta)}f\right]\mu
(d\beta)\\
&=\lim_{n\rightarrow \infty} L\sum^{N_{n}}_{j=1}\alpha_{n,j}\frac{e^{c_{n_{j}}t}}{c_{n_{j}}}
\int^{\infty}_{0}e^{-C_{n_{j}}t}S(t)fdt.
\end{split}
\end{equation}
Using again the fact that $L$ is closed we get
$$\int^{1}_{0}\left[\int^{t}_{0}\frac{(t-u)^{-\beta}}{\Gamma(1-\beta)}S(u)fdu-\frac{t^{1-\beta}}{\Gamma(2-\beta)}f\right]\mu
(d\beta)=L\int^{t}_{0}S(u)fdu$$
and now take the derivative with respect to $t$ on both sides to obtain
 $$\int^{1}_{0}\left(\frac{d^{\beta}}{dt^{\beta}}S(t)f-\frac{t^{-\beta}}{\Gamma(1-\beta)}f\right)\mu
(d\beta)=LS(t)f$$
for all $t>0$, where we use the fact that $L$ is closed to justify taking the derivative inside.
\end{proof}
\begin{corollary}
Let $0<\gamma\leq 2$. Let $-(-\Delta)^{\gamma/2}$ be fractional Laplacian on $L^1(\R^d)$ corresponding to the semigroup $T(t)$ on $L^1(\R^d)$. Let $Y(t)$ be the corresponding symmetric stable process (i.e. $T(t)f(x)=\E_x(f(Y(t)))$ ).  Then the family $\{S(t):t\geq 0\}$ of linear operators from $X$ into $X$ given by $S(t)f=\displaystyle
\int^{\infty}_{0}(T(l)f)f_{E_\mu(t)}(l)dl=\E(f(Y(E_\mu(t))))$, is uniformly bounded and strongly analytic in a sectorial region. Furthermore,
$\{S(t):t\geq 0\}$ is strongly continuous and $g(x, t)=S(t)f(x)$ is a solution of \begin{equation}\label{distributed-frac-prob-frac-laplace}
\int^{1}_{0}\partial^{\beta}_{t}g(x,t)\mu(d\beta)= -(-\Delta)^{\gamma/2}g(x,t); \ g(x,0)=f(x).
\end{equation}

\end{corollary}

\begin{remark}
It looks like a challenging problem to extend the methods applied in the main results to more general  time operators defined as $\psi(\partial_t)-\phi(t,\infty)$ where $\psi$ is defined in \eqref{psiD2}. Meerschaert and Scheffler \cite{M-S-triangular} define this operator by its Laplace transform as
$$\int_0^\infty e^{-st}\psi(\partial_t)g(t)dt=\psi(s)\tilde{g}(s).$$
Likewise it looks like a challenging problem to extend the results in this paper to the case where $\psi(s)=(s+\lambda)^\beta-\lambda^\beta$ for $\lambda>0$. This  $\psi$ gives rise to the so called tempered fractional  derivative operator studied by \cite{temperedLM, m-n-v-pams}.
\end{remark}

\vspace{1cm}
\noindent {\bf Acknowledgements}. The authors are grateful to Mark M. Meerschaert  for suggesting the problems studied in this paper, and for fruitful discussion about the results in this paper.





\end{document}